\documentclass[11pt]{amsart}

\usepackage[latin1]{inputenc} 
\usepackage{amsthm,amsmath,amssymb,bm,mathtools} 
\usepackage{enumerate}
\usepackage[foot]{amsaddr}

\theoremstyle{plain}
\newtheorem{theorem}{Theorem}[section]
\newtheorem{proposition}[theorem]{Proposition}
\newtheorem{lemma}[theorem]{Lemma}
\newtheorem{fact}[theorem]{Fact}
\theoremstyle{definition}
\newtheorem{example}[theorem]{Example}
\newcommand{\sm}{\setminus}
\newcommand{\eps}{\varepsilon}

\title{Matchings in multipartite hypergraphs}
\author{Candida Bowtell}
\address{School of Mathematics, University of Birmingham, Birmingham B15 2TT, United Kingdom}
\email{c.bowtell@bham.ac.uk}
\thanks{CB gratefully acknowledges support from ERC Advanced Grant 883810, ERC Starting Grant 947978 and Leverhulme Trust Early Career Fellowship ECF--2023--393; RM gratefully acknowledges support from EPSRC Standard Grant EP/R034389/1.}
\author{Richard Mycroft}
\email{r.mycroft@bham.ac.uk}

\begin{document}

\date{}

\begin{abstract} 
A folklore result on matchings in graphs states that if $G$ is a bipartite graph whose vertex classes $A$ and $B$ each have size $n$, with $\deg(u) \geq a$ for every $u \in A$ and $\deg(v) \geq b$ for every $v \in B$, then $G$ admits a matching of size $\min\{n, a+b\}$. In this paper we establish the analogous result for large $k$-partite $k$-uniform hypergraphs, answering a question of Han, Zang and Zhao, who previously demonstrated that this result holds under the additional condition that the minimum degrees into at least two of the vertex classes are large. A key part of our proof is a study of rainbow matchings under a combination of degree and multiplicity conditions, which may be of independent interest.
\end{abstract}

\maketitle

\section{Introduction}

 An old folklore result, whose proof often appears as a straightforward exercise in elementary graph theory courses, states 
the following.
\begin{fact} \label{thm_for_graphs}
Let~$G$ be a bipartite graph whose vertex classes~$A$ and~$B$ contain~$n$ vertices each. If $\deg(u) \geq a$ for every vertex $u \in A$, and $\deg(v) \geq b$ for every vertex $v \in B$, then~$G$ admits a matching (i.e. a set of pairwise-disjoint edges) of size $\min\{n, a+b\}$.
\end{fact}
This bound on the size of the matching is best possible for all~$a, b$ and~$n$, in the sense that it is not possible to have a matching of size larger than~$n$, and moreover there exist graphs satisfying the degree conditions which do not contain a matching of size larger than~$a+b$.

The central result of this paper is to prove the analogous statement for large multipartite hypergraphs. To explain this precisely, we make the following definitions, all of which are standard in extremal graph theory. A \emph{$k$-uniform hypergraph}, or simply \emph{$k$-graph}, consists of a vertex set~$V$ and an edge set~$E$, where each edge is a set of~$k$ vertices. So, in particular, a $2$-graph is a simple graph. Exactly as for graphs, a \emph{matching} in a $k$-graph~$H$ is a set of pairwise-disjoint edges, and we write~$\nu(H)$ for the \emph{matching number} of~$H$, that is, the size of a largest matching in~$H$. A matching in~$H$ is \emph{perfect} if it covers every vertex of~$H$. We say that a $k$-graph is \emph{$k$-partite} if its vertex set~$V$ admits a partition $V = V_1 \cup \dots \cup V_k$ for which every edge $e \in E$ has precisely one vertex in each set~$V_i$; we refer to the sets~$V_i$ as \emph{vertex classes}. Observe that a $2$-partite $2$-graph is exactly a bipartite graph. Given a $k$-partite $k$-graph~$H$ with vertex classes $V_1, \ldots, V_k$, we say that a set $S \subseteq V(H)$ is \emph{crossing} if it has at most one vertex in each vertex class, and that~$S$ \emph{avoids}~$V_i$ if $S \cap V_i = \emptyset$. For a crossing $(k-1)$-tuple~$S$ of vertices of~$V(H)$, we define the \emph{degree}~$\deg(S)$ of~$S$ to be the number of edges $e \in E(H)$ with $S \subseteq e$ (in other words, the number of ways to add a vertex to~$S$ to form an edge of~$H$; note that since~$H$ is $k$-partite this vertex must come from the vertex class which~$S$ avoids). Finally, for each $i \in [k]$ we define the \emph{minimum multipartite codegree into~$V_i$}, denoted~$\delta_{[k] \sm \{i\}}(H)$, to be the minimum of~$\deg(S)$ over all crossing $(k-1)$-tuples~$S$ of vertices of~$H$ which avoid~$V_i$. We can now give the statement of our main result.

\begin{theorem} \label{thm_main}
For each $k \geq 3$, there exists~$n_0$ for which the following statement holds. Let~$H$ be a $k$-partite $k$-graph with vertex classes $V_1, \dots, V_k$ each of size $n \geq n_0$ in which $\delta_{[k]\setminus \{i\}}(H) \geq a_i$ for each $i \in [k]$. Then 
$$\nu(H)\geq \min\{n-1, \sum_{i=1}^k a_i\}.$$
\end{theorem}

One might initially wonder whether we could hope for the stronger conclusion that 
$\nu(H) \geq \min\{n, \sum_{i \in [k]} a_i\}$ in Theorem~\ref{thm_main}. However, Theorem~\ref{thm_main} was shown to be best possible in this sense by the following extremal construction presented by Han, Zang and Zhao~\cite{main}.

\begin{example}[Divisibility barrier] \label{ex_divis} 
Take sets $V_i:=A_i \dot{\cup} B_i$ with $n/2-1 \leq |A_i|\leq n/2+1$ for each $i \in [k]$, ensuring that $\sum_{i \in [k]} |A_i|$ is odd. Let~$H$ be the $k$-partite $k$-graph with vertex classes $V_1, \dots, V_k$ whose edges are all crossing $k$-tuples containing an even number of vertices in $\bigcup_{i \in [k]} A_i$. We then have $\delta_{[k]\setminus \{i\}}(H) \geq a_i := n/2-1$ for each~$i$, so $\sum_{i \in [k]} a_i \geq k(n/2-1)$. However, every matching~$M$ in~$H$ contains an even number of vertices in $\bigcup_{i \in [k]} A_i$, so is not perfect since $\sum_{i \in [k]} |A_i|$ is odd. 
\end{example}

The same authors also presented the following extremal construction showing that the conclusion of Theorem~\ref{thm_main} cannot be strengthened in the case when $\sum_{i \in [k]} a_i \leq n-1$.

\begin{example}[Space barrier]\label{ex_space}
Take sets $A_i \subseteq V_i$ with $|A_i|=a_i$ and $|V_i| = n$ for each $i \in [k]$, and let~$H$ be the $k$-partite $k$-graph with vertex classes $V_1, \dots, V_k$ whose edges are all crossing $k$-tuples which intersect $\bigcup_{i \in [k]} A_i$. We then have $\delta_{[k]\setminus \{i\}}(H) \geq a_i$ for each $i \in [k]$, but every matching~$M$ in~$H$ then has $|M| \leq |\bigcup_{i \in [k]} A_i| = \sum_{i \in [k]} a_i$ since every edge of~$M$ contains a vertex from $\bigcup_{i \in [k]} A_i$.
\end{example}

Previous work on this question established many cases of Theorem~\ref{thm_main}. Indeed, Han, Zang and Zhao~\cite{main} observed that Theorem~\ref{thm_main} may be proved in the case when $\sum_{i \in [k]} a_i \leq n-k+2$ by adapting a short elegant argument of K\"uhn and Osthus~\cite{kuhn-ost} which gave a matching of size~$n-k+2$ in the case in which $a_i = n/k$ for each $i \in[k]$ by showing that any smaller matching can be enlarged by removing an edge and adding two new edges. 

\begin{fact} \label{prop_fact}
Let~$H$ be a $k$-partite $k$-graph with vertex classes $V_1, \dots, V_k$ each of size~$n$. If $\delta_{[k]\setminus \{i\}} (H) \geq a_i$ for every $i \in [k]$, then
$$\nu(H) \geq \min\{n-k+2, \sum_{i=1}^k a_i\}.$$
\end{fact}

In the same paper, Han, Zang and Zhao used a sophisticated absorbing argument to prove Theorem~\ref{thm_main} in all cases where at least two of the minimum multipartite codegrees are large.

\begin{theorem}\cite[Theorem 1.1]{main} \label{thm_HZZ}
For each~$k \geq 3$ and all~$\eps > 0$ there exists $n_0 \in \mathbb{N}$ for which the following statement holds. Let~$H$ be a $k$-partite $k$-graph with vertex classes $V_1, \dots, V_k$ each of size $n \geq n_0$ in which $\delta_{[k]\setminus \{i\}}(H) \geq a_i$ for each $i \in [k]$. If $a_1 \geq a_2 \geq \ldots \geq a_k$ and $a_2 > \eps n$, then $$\nu(H) \geq \min\{n-1, \sum_{i=1}^{k} a_i\}.$$
\end{theorem}

It remained to establish Theorem~\ref{thm_main} in the cases where all but one of the minimum multipartite codegrees are small, and Han, Zang and Zhao remarked that they were not sure whether the statement extends to the remaining cases. In this paper we show that this is indeed the case, and in fact we obtain a slightly stronger bound in the remaining cases, as stated in the following theorem.

\begin{theorem} \label{thm_extremal_n-1}
For each $k \geq 3$ there exists $n_0 \in \mathbb{N}$ for which the following statement holds. Let~$H$ be a $k$-partite $k$-graph with vertex classes $V_1, \dots, V_k$ each of size $n \geq n_0$ in which $\delta_{[k]\setminus \{i\}}(H) \geq a_i$ for each $i \in [k]$. If $a_1 \geq a_2 \geq \ldots \geq a_k$ and $\sum_{i \in [2,k]} a_i \leq \frac{n}{1600k^4}$, then $$\nu(H) \geq \min\{n, \sum_{i=1}^{k} a_i\}.$$
\end{theorem}

Taking $\eps = 1/1600k^5$, Theorems~\ref{thm_HZZ} and~\ref{thm_extremal_n-1} together prove Theorem~\ref{thm_main}, since by relabelling the vertex classes if necessary we may assume without loss of generality that $a_1 \geq a_2 \geq \ldots \geq a_k$.

\subsection{Wider context}

Determining the size of a largest matching in a given $k$-graph is a key question in extremal combinatorics which has been a central focus of research activity over recent decades. As well as being a natural, fundamental question in its own right, this question arises since matchings provide a general framework for many important problems within combinatorics such as combinatorial designs, Latin squares, tilings and decompositions, as well as in practical applications such as the `Santa Claus' allocation problem~\cite{AFS} and for distributed storage allocation~\cite{AFHRRS}.
In the graph case Edmond's blossom algorithm~\cite{E} gives an efficient method for computing a largest matching, whilst Hall's marriage theorem~\cite{hall} and Tutte's theorem~\cite{tutte} give elegant characterisations of bipartite graphs and arbitrary graphs respectively which contain perfect matchings. By contrast, for $k \geq 3$ the problem of determining the size of a largest matching in a $k$-graph is an NP-hard problem, as is the more restricted question of whether a given 3-partite 3-graph contains a perfect matching; the latter problem is often referred to as \emph{3-dimensional matching} and was one of Karp's original 21 NP-complete problems~\cite{karp}.  Consequently we do not expect to find similar characterisations as for the graph case. Instead, the principal direction of research has been to identify sufficient conditions which guarantee the existence of a matching of size~$s$ in a $k$-graph~$H$, with particular focus on the case of a perfect matching, i.e. when $s = |V(H)|/k$. Since a necessary condition for a perfect matching is that~$H$ has no isolated vertices, the most natural parameters to consider are minimum degree conditions, also known as Dirac-type conditions after Dirac's celebrated theorem~\cite{dirac} that every graph~$G$ on $n \geq 3$ vertices with minimum degree $\delta(G) \geq n/2$ contains a Hamilton cycle (a cycle covering every vertex); for even~$n$, choosing every other edge from such a cycle gives a perfect matching in~$G$. (Another recent research focus has been to find minimum degree conditions under which the perfect matching problem in $k$-graphs is computationally tractable~\cite{han2, HT, KRS, divbarrier, Szy}.)

The strongest and most frequently used minimum degree conditions for $k$-uniform hypergraphs~$H$ are \emph{codegree} conditions, that is, conditions on the degree,~$\deg(S)$, of sets $S \subseteq V(H)$ of size~$k-1$; the usefulness of such conditions is that for such sets~$\deg(S)$ indicates the number of vertices which can be added to~$S$ to form an edge of~$H$. Recall from the opening section that for a $k$-partite~$k$-graph~$H$ with vertex classes $V_1, \dots, V_k$ the minimum multipartite codegree into~$V_i$, denoted~$\delta_{[k]\sm \{i\}}(H)$, is the minimum of~$\deg(S)$ over all crossing $(k-1)$-tuples which avoid~$V_i$; in the same way we define the minimum multipartite codegree of~$H$, denoted~$\delta^*_{k-1}(H)$, to be the minimum of~$\deg(S)$ over all crossing $(k-1)$-tuples. For general (i.e. not necessarily $k$-partite) $k$-graphs~$H$ we likewise define the \emph{minimum codegree of~$H$}, denoted~$\delta_{k-1}(H)$, to be the minimum of~$\deg(S)$ over all sets of~$k-1$ vertices. 

K\"{u}hn and Osthus~\cite{kuhn-ost} gave asymptotically best-possible sufficient minimum codegree conditions for a perfect matching in both settings, specifically $\delta^*_{k-1}(H) > n/2 + O(\sqrt{n\log(n)})$ in the multipartite setting (with~$n$ being the common size of the vertex classes) and $\delta_{k-1}(H) > n/2 + O(\sqrt{n\log(n)})$ in the non-partite setting (where $n = |V(H)|$ and we assume the necessary condition that~$k$ divides~$n$ without further comment). For the general setting R\"{o}dl, Ruci\'{n}ski and Szemer\'{e}di~\cite{rrs1} improved this result to the exact best-possible minimum codegree condition for large values of~$n$, namely $\delta_{k-1}(H)\geq n/2-k+c$ where $c \in \{1/2, 1, 3/2, 2\}$ depends on the values of~$n$ and~$k$. Similarly, in the multipartite setting Aharoni, Georgakopoulos and Spr\"{u}ssel~\cite{ags} showed that (for large~$n$) if~$H$ is a $k$-partite $k$-graph such that $\delta_{[k]\setminus\{1\}}(H)>n/2$ and $\delta_{[k]\setminus\{2\}}(H)\geq n/2$, then (with no further conditions on the minimum multipartite codegree into parts $V_3, \ldots, V_k$)~$H$ must contain a perfect matching; in particular this implies that the condition $\delta^*_{k-1}(H) > n/2$ is a best-possible sufficient condition for the existence of a perfect matching in~$H$. 

Turning now to smaller (i.e. non-perfect) matchings, K\"{u}hn and Osthus~\cite{kuhn-ost} gave a short elegant argument considering the size of a largest matching to show that if~$H$ is a $k$-partite $k$-graph with vertex classes each of size~$n$, and $\delta^*_{k-1}(H) \geq n/k$, then~$H$ contains a matching of size~$n-k+2$. It follows immediately from their argument that for $s \leq n-k+2$, if $\delta^*_{k-1}(H) \geq \lceil{s/k}\rceil$ then~$H$ contains a matching of size~$s$. R\"{o}dl, Ruci\'{n}ski and Szemer\'{e}di~\cite{rrs1} adapted this argument to the general setting, showing that for $s \leq \lfloor {n/k} \rfloor -k+2$ every $k$-graph~$H$ with $\delta_{k-1}(H) \geq s$ has a matching of size~$s$. These results left open the correct condition for a matching of size~$s$ in the cases $n-k+3 \leq s < n$ (in the multipartite setting) or size $\lfloor {n/k} \rfloor -k+3 \leq s < n/k$ (in the general setting). For the general setting this question was answered asymptotically by R\"{o}dl, Ruci\'{n}ski and Szemer\'{e}di~\cite{rrs2} before Han~\cite{han} gave an exact answer for large~$n$, showing that $\delta_{k-1}(H) \geq s$ is a best-possible sufficient condition for these cases also. In the multipartite setting R\"{o}dl and Ruci\'{n}ski~\cite[Problem 3.14]{rodl_ruc} asked, for a $k$-partite $k$-graph~$H$ with vertex classes each of size~$n$, whether the condition $\delta^*_{k-1}(H) \geq \lceil n/k \rceil$ was sufficient for a matching of size~$n-1$, and this was answered in the affirmative by Lu, Wang and Yu~\cite{lwy} and independently by Han, Zang and Zhao \cite{main}; the latter authors also showed that $\delta^*_{k-1}(H)\geq \lfloor n/k \rfloor$ suffices in the case where~$n \equiv 1 \mod k$. This established the correct condition for~$s=n-1$, and the correct condition for the remaining values of~$s$, namely that $\delta^*_{k-1}(H) \geq \lceil s/k \rceil$ is a best-possible sufficient condition, follows immediately from the proof given by Han, Zang and Zhao. Note that Theorem~\ref{thm_main} gives a wide-ranging generalisation of these multipartite results by giving a much broader collection of codegree conditions which ensure the existence of a matching of size~$n-1$. 

Given a $k$-graph~$H$ it is also natural to consider conditions on the degree,~$\deg(T)$, of sets $T \subseteq V(H)$ of fewer than~$k-1$ vertices (which is again defined to be the number of edges containing~$T$). For each $1 \leq t \leq k-1$ we define the \emph{minimum $t$-degree} of~$H$, denoted~$\delta_t(H)$, to be the minimum of~$\deg(T)$ over all sets $T \subseteq V(H)$ of size~$t$; in particular~$\delta_1(H)$ is also commonly referred to as the minimum vertex degree of~$H$. Much work~\cite{AFHRRS, bde, CGHW, CK, DH, HPS, han3, Kh2, Kh, kot, kot2, LYY, MR2, TZ} has focused on identifying best-possible conditions on~$\delta_t(H)$ which suffice to ensure the existence of a perfect matching (or a matching of a given size~$s$) in~$H$ for~$t \leq k-2$, however for~$t < 0.42k$ the problem remains open in general.

Turning back to the multipartite setting in which~$H$ is a $k$-partite $k$-graph whose vertex classes each have size~$n$, for the case~$k=3$, Lo and Markstr\"{o}m \cite{lo-mark} gave the exact best-possible minimum vertex degree condition to ensure a perfect matching in a $3$-partite $3$-graph for sufficiently large~$n$, as well as the exact best-possible minimum vertex degree condition to ensure a matching of size~$m+1$ for every~$m$ when~$n \geq 3^7m$. Pikhurkho instead considered an inhomogeneous type of degree condition, writing $\delta_{I}(H):=\min\{\deg(S): |S \cap V_j|=1 \mbox{ if and only if } j \in I\}$ for each set $I \subseteq [k]$; he showed that if there exists $\emptyset \neq L \subsetneq [k]$ with $\frac{\delta_L(H)}{n^{k-|L|}}+\frac{\delta_{[k]\setminus L}(H)}{n^{|L|}} = 1+\Omega\left(\sqrt{\frac{\log(n)}{n}}\right)$ then~$H$ contains a perfect matching. Keevash and Mycroft~\cite{KM} (see also Han~\cite{han4}) also gave degree sequence conditions -- where an edge may be constructed vertex-by-vertex with a given number of options at each step -- which suffice to ensure a perfect matching in this setting.

We recommend the surveys of K\"uhn and Osthus~\cite{kosurvey}, R\"odl and Ruci\'nski~\cite{rodl_ruc} and Zhao~\cite{zhao} for further information on this and related problems.

\subsection{Main ideas} \label{sec:ideas} As in Theorems~\ref{thm_HZZ} and~\ref{thm_extremal_n-1}, let~$H$ be a $k$-partite $k$-graph with vertex classes $V_1, \dots, V_k$ each of size $n \geq n_0$ in which $\delta_{[k]\setminus \{i\}}(H) \geq a_i$ for each $i \in [k]$, where $a_1 \geq a_2 \geq \ldots \geq a_k$. At first glance it may seem surprising that the cases of Theorem~\ref{thm_main} which remained open are those in which $a_2, \dots, a_k$ are all small, since in light of Fact~\ref{prop_fact} we may then assume that~$a_1$ is very close to~$n$, meaning that~$H$ is extremely dense. However, it is possible in this case that the condition that $\delta_{[k]\setminus \{1\}}(H) \geq a_1$ is satisfied by the existence of a set $S \subseteq V_1$ of size~$a_1$ whose vertices form edges with all crossing~$(k-1)$-tuples which avoid~$V_1$. This makes it trivial to find a matching covering all vertices of~$S$, but such a matching is still not large enough to meet the requirement of Theorem~\ref{thm_main}. Consequently our matching must also contain edges which do not intersect~$S$, but the fact that $a_2, \dots, a_k$ are all small means that the $k$-graph formed by such edges can be very sparse. For instance, if $a_2, \dots, a_k$ are all of constant size, then the number of such edges could be of order~$O(n^{k-1})$. The central difficulty in proving Theorem~\ref{thm_extremal_n-1}  is to construct a sufficiently large matching in this sparse environment. 

\medskip \noindent {\bf Rainbow matchings under a combination of degree and multiplicity.} 
Most of the work in proving Theorem~\ref{thm_extremal_n-1} is contained in Section~\ref{sec:comb}, where we consider large rainbow matchings in families of $k$-partite $k$-graphs which satisfy a combination of degree and multiplicity conditions. For this we make the following definitions: let $\mathcal{H}:=\{H_1, \ldots, H_t\}$ be a family of $k$-graphs on the same vertex set~$V$. A \emph{rainbow matching} in~$\mathcal{H}$ is a set~$M$ of pairwise-disjoint $k$-tuples of vertices of~$V$ such that there is an injective map $\phi: M \to [t]$ for which $e \in E(H_{\phi(e)})$ for every~$e \in M$. In other words,~$M$ is a matching within the union of the $k$-graphs in~$\mathcal{H}$ in which each edge of~$M$ is taken from a distinct member of~$\mathcal{H}$. Implicitly fixing such an injection~$\phi$, we say that the edge~$e \in M$ has \emph{colour}~$j$ if and only if~$\phi(e) = H_j$ (so the colour of~$e$ is the unique $k$-graph in~$\mathcal{H}$ from which it is taken, not simply a $k$-graph in~$\mathcal{H}$ which contains~$e$, of which there may be many). We say that a rainbow matching in~$\mathcal{H}$ is \emph{perfect} if it has size~$t$ (the size of~$\mathcal{H}$), meaning that~$\mathcal{H}$ has precisely one edge of each colour. 

The scenario we work with in Section~\ref{sec:comb} is that we have a family $\mathcal{H}=\{H_1, \ldots, H_t\}$ of $k$-partite $k$-graphs on common vertex classes $V_1, \dots, V_k$ each of size~$n$ which satisfies, for some integers $a_1, \dots, a_k$ and~$m$, the conditions that
\begin{enumerate}[(i)]
    \item \label{sketchpropi} for every $j \in [t]$ and $i \in [k]$ we have $\delta_{[k]\sm \{i\}}(H_j) \geq a_i$, and
    \item \label{sketchpropii} every crossing $k$-tuple of vertices is an edge of at least~$m$ of the $k$-graphs in $\mathcal{H}$.
\end{enumerate}
We describe in the next subsection how such a family arises naturally in the context of Theorem~\ref{thm_extremal_n-1} as a collection of link graphs of vertices in the first vertex class of~$H$. The key parameters we work with are~$m$ and $q := \sum_{i \in [k]} a_i$.
Indeed, the key step in the proof of Theorem~\ref{thm_extremal_n-1} is Lemma~\ref{lem:perfect_rainbow_matching}, which states that --- provided~$m$ is small relative to~$q$, and assuming the necessary condition that~$t \geq m+q$ --- conditions~(\ref{sketchpropi}) and~(\ref{sketchpropii}) are sufficient to guarantee a rainbow matching in~$\mathcal{H}$ of size~$m+q$ (for~$n$ sufficiently large).

Whilst this result is sufficient for our purposes in this paper, we suggest that conditions~(\ref{sketchpropi}) and~(\ref{sketchpropii}) should in fact guarantee a rainbow matching of size~$m+q$ for a much wider range of values of~$m$ and~$q$; given the natural way in which conditions~(\ref{sketchpropi}) and~(\ref{sketchpropii}) arise (described below) we hope that this question may be of independent interest. On the other hand, we cannot hope for these conditions to ensure a rainbow matching of size larger than~$m+q$; to see this, let $H_1, \dots, H_m$ be complete $k$-partite $k$-graphs, whilst $H_{m+1}, \dots, H_t$ are as in Example~\ref{ex_space}. We then have no rainbow matching in~$\mathcal{H}$ with size larger than~$m+q$, since no rainbow matching in $\{H_{m+1}, \dots, H_t\}$ has size larger than $\sum_{i \in [k]} a_i = q$.

For~$m=0$ (meaning that condition~(\ref{sketchpropii}) provides no information) and~$q \geq t$ Pokrovskiy~[personal communication] gave an elegant induction argument to show that a perfect rainbow matching (i.e. of size~$t$) must exist; this argument was the inspiration for our Lemma~\ref{lem:rainbow_or_domset}, and we include it here for completeness (though we do not use it in our proof). For a $k$-graph~$H$ we say that a set $U \subseteq V(H)$ is a {\it dominating set} if $U \cap e \neq \emptyset$ for every $e \in E(H)$.

\begin{lemma} \label{lem_alexey}
Let $V_1, V_2, \dots, V_k$ be disjoint sets of vertices, and let $\mathcal{H}:=\{H_1,\ldots,H_t\}$ be a family of $k$-partite $k$-graphs each with vertex classes $V_1, \ldots, V_k$ such that for every $i \in [k]$ and $j \in [t]$ we have $\delta_{[k]\sm \{i\}}(H_j) \geq a_i$. If $t \leq \sum_{i=1}^k a_i \leq \frac{\min_{i \in [k]}|V_i|}{k^{10}}$, then~$\mathcal{H}$ has a perfect rainbow matching.
\end{lemma}

\begin{proof}
We proceed by induction on~$t$; the case~$t=0$ is trivial, so suppose that the lemma holds for families of size~$t-1$. Without loss of generality, we may assume that $t=\sum_{i=1}^k a_i$ and that $a_1 \geq \frac{t}{k}$.

{\bf{Case 1:}}~$H_t$ contains a matching~$M$ of size~$kt-k+1$. By induction applied to $\mathcal{H'}=\{H_1, \ldots, H_{t-1}\}$, the family~$\mathcal{H}'$ admits a rainbow matching~$M'$ of size~$t-1$. Since~$M$ has size~$kt-k+1$ and $|V(M')| = k(t-1)$, there must exist an edge~$e \in M$ with $e \cap V(M') = \emptyset$, and then~$M' \cup \{e\}$ is a rainbow matching for $\mathcal{H}=\{H_1, \ldots, H_{t}\}$ of size~$t$, as required.

{\bf{Case 2:}}~$H_t$ contains no matching~$M$ of size~$kt-k+1$. It follows that the vertex set of a maximum matching in~$H_t$ is a dominating set~$D$ of order strictly less than~$k^2t$. Let $e_1, \ldots, e_{k^5t}$ be an arbitrary family of pairwise-disjoint crossing $(k-1)$-tuples avoiding~$V_1$ which do not intersect~$D$. This is possible since~$D$ intersects each~$V_i$ in at most~$kt$ vertices and $|V_i|-kt \geq k^{10}t - kt \geq k^5t$ for each $i \in [k]$. The degree condition implies that for each $j \in [k^5t]$ the $(k-1)$-tuple~$e_j$ has at least~$a_1$ neighbours in~$H_t$ (that is, $|N_{H_t}(e_j)| \geq a_1$). Since every edge in~$H_t$ intersects~$D$, and~$e_j$ is disjoint from~$D$, we have $N_{H_t}(e_j) \subseteq D$ for each $j \in [k^5t]$. So by averaging, some vertex $v \in V_1$ is a neighbour of at least $k^5ta_1/|D| \geq k^2t$ of the $(k-1)$-tuples~$e_j$, where the inequality holds since $a_1 \geq t/k$ and $|D| \leq k^2t$. By relabelling the~$e_j$ if necessary we may assume that $v \in \bigcap_{j \in [k^2t]} N_{H_t}(e_j)$. Let $V' = \bigcup_{i \in [k]} V_i \sm \{v\}$ and let $H'_j = H_j[V']$ for each $j \in [t-1]$; by induction the family $\mathcal{H'} = \{H'_1, \ldots, H'_{t-1}\}$ then admits a rainbow matching~$N$ of size~$t-1$ which, by construction, does not cover the vertex~$v$. Furthermore, $|V(N)| = k(t-1) < k^2t$ and so there exists $j \in [k^2t]$ for which $e_j \cap V(N) = \emptyset$. It follows that $M = N \cup \{\{v\} \cup e_j\}$ is a rainbow matching for $\mathcal{H}=\{H_1, \ldots, H_{t}\}$ of size~$t$, as required.
\end{proof}

\medskip \noindent {\bf Outline proof of Theorem~\ref{thm_extremal_n-1}.} To illustrate how the result for rainbow matchings described above can be used to prove Theorem~\ref{thm_extremal_n-1}, we consider here the case of Theorem~\ref{thm_extremal_n-1} in which $\sum_{i \in [k]} a_i = n$; the argument for the other cases not covered by Fact~\ref{prop_fact} is broadly similar. In this case we wish to find a perfect matching in~$H$. Write $q := \sum_{2 \leq i \leq k} a_i = n-a_1$, so the assumption of Theorem~\ref{thm_extremal_n-1} implies that~$q$ is small. We choose an arbitrary perfect matching~$M$ in the complete $(k-1)$-partite $(k-1)$-graph on $V_2, \dots, V_k$, and use a double-counting argument (Proposition~\ref{lemma_extremal_count}) to show that there is a subset $X \subseteq V_1$ of size close to~$n-q$ such that every vertex in~$X$ forms an edge with many $(k-1)$-tuples in~$M$. Let $Z = V_1 \sm X$ be the set of the remaining vertices of~$V_1$, so~$|Z|$ is slightly larger than~$q$. 

For each vertex $v \in Z$ let~$L_v(H)$ be the \emph{link graph} of~$v$; this is the $(k-1)$-partite $(k-1)$-graph with vertex classes $V_2, \dots, V_k$ whose edges are all crossing $(k-1)$-tuples~$f$ for which~$\{v\} \cup f \in E(H)$. The assumption that $\delta_{[k]\setminus \{i\}}(H) \geq a_i$ for $i \in [2, k]$ yields a similar condition on the link graph~$L_v(H)$ for each $v \in Z$, namely that for each $i \in [2, k]$, each crossing $(k-2)$-tuple avoiding~$V_i$ has at least~$a_i$ neighbours in~$V_i$. Moreover, the assumption that $\delta_{[k] \setminus \{1\}}(H) \geq a_1$ yields a multiplicity condition on the link graphs~$L_v(H)$, namely that every crossing $(k-1)$-tuple~$f$ avoiding~$V_1$ appears in at least~$a_1$ of the link graphs~$L_v(H)$ of vertices $v \in V_1$; it follows that~$f$ appears in at least $m := |Z| - (n-a_1) = |Z| - q$ of the link graphs~$L_v(H)$ of vertices $v \in Z$ (see Proposition~\ref{prop_complete}). In other words, taking $U_i := V_{i+1}$ for each $i \in [k-1]$ the family $\mathcal{H} = \{H_1, \dots, H_{|Z|}\} = \{L_v(H) : v \in Z\}$ of link graphs of vertices of~$Z$ is a family of $(k-1)$-partite $(k-1)$-graphs with common vertex classes $U_1, \dots U_{k-1}$ each of size~$n$ which satisfies conditions~(\ref{sketchpropi}) and~(\ref{sketchpropii}) above (with the same~$q$ there as here and with~$m = |Z| - q$). So by Lemma~\ref{lem:perfect_rainbow_matching} (described above) there exists a rainbow matching in~$\mathcal{H}$ of size~$m+q = |Z|$. 

Taking each edge of this matching together with the corresponding vertex of~$Z$ gives a matching~$M_{rainbow}$ in~$H$ of size~$|Z|$ which covers every vertex of~$Z$. Using the fact that vertices in~$X$ form edges with many $(k-1)$-tuples in~$M$, and that crossing $(k-1)$-tuples avoiding~$V_1$ have many neighbours in~$V_1$, it is then straightforward to extend~$M_{rainbow}$ to a perfect matching in~$H$ (see the proof of Theorem~\ref{thm_extremal_n-1} in Section~\ref{sec:proof} for the details).

\subsection{Notation}

We write $[n]:=\{1,2,\ldots,n\}$ and $[a,b]:=\{a, a+1, \ldots, b\}$ for all~$b >a$ and $a,b,n \in \mathbb{N}$. If~$H$ is a $k$-graph with~$k \geq 3$, then for a $(k-1)$-tuple of vertices~$S \subseteq V(H)$, we write $N_H(S):=\{v \in V(H): S \cup \{v\} \in E(H)\}$ for the \emph{neighbourhood} of~$S$ in~$H$, so $\deg_H(S) = |N_H(S)|$. On the other hand, if~$G$ is a graph then for a set of vertices $S \subseteq V(G)$ we write~$N_G(S) :=  \bigcup_{x \in S} N_G(x)$; there should be no confusion as the former notation will never be used in the graph case. In both cases we may omit the subscript when the host graph is clear from the context.

Throughout the paper, in any condition of the form $\delta_{[k]\setminus \{i\}}(H) \geq a_i$ it should be assumed that~$a_i$ is a non-negative integer.

\section{Matchings under a combination of degree and multiplicity} \label{sec:comb}

In this section we prove the key lemma we need for the proof of Theorem~\ref{thm_extremal_n-1}, namely Lemma~\ref{lem:perfect_rainbow_matching}. As described in Section~\ref{sec:ideas}, the setting for this is that we have a family $\mathcal{H} = \{H_1, \dots, H_t\}$ of $k$-partite $k$-graphs on common vertex classes $V_1, \dots, V_k$, and each $k$-graph~$H_j$ in~$\mathcal{H}$ satisfies the multipartite codegree condition that $\delta_{[k]\setminus\{i\}}(H_j) \geq a_i$ for each $i \in [k]$. We write $q := \sum_{i=1}^k a_i$, and this is the key parameter concerning degrees in each~$H_j$. Our key lemma, Lemma~\ref{lem:perfect_rainbow_matching}, then says that if we also assume the multiplicity condition that each crossing $k$-tuple is contained in at least~$m$ of the $k$-graphs in~$\mathcal{H}$, and the necessary condition that~$t \geq m+q$, then (under some additional technical assumptions) there must be a rainbow matching in~$\mathcal{H}$ of size~$m+q$. We begin working towards this goal by proving the following weaker version of the lemma, in which we assume that~$t \geq m+q+k-1$, meaning that our matching may leave~$k-1$ colours unused; the lemma allows us some control over which colours these are. Note that here we also allow the sizes of vertex classes to differ from one another, to enable the lemma to be applied after a small number of vertices have been deleted from our original vertex classes.

\begin{lemma} \label{lem:almostperfectrainbowmatching}
Let $H_1, \dots, H_t$ be $k$-partite $k$-graphs with common vertex classes $V_1, \dots, V_k$ each of size at least~$n$. Suppose that
\begin{enumerate}[(i)]
    \item \label{lem:almostperfectrainbowmatching_i} for each~$i \in [k]$ and~$j \in [t]$ we have $\delta_{[k]\setminus\{i\}}(H_j) \geq a_i$, and
    \item \label{lem:almostperfectrainbowmatching_ii} each crossing $k$-tuple appears in at least~$m$ of the $k$-graphs~$H_j$.
\end{enumerate}
Also write $q := \sum_{i \in [k]} a_i$. If $n, t \geq m + q + k - 1$ then for each set $C \subseteq [t]$ with size $|C| \leq q/k$ there is a rainbow matching in $\mathcal{H} := \{H_j : j \in [t]\}$ of size~$m+q$ which includes an edge of colour~$j$ for each~$j \in C$.
\end{lemma}

\begin{proof}  We first show that~$\mathcal{H}$ admits a rainbow matching of size~$m+q$. Let~$M$ be a largest rainbow matching in~$\mathcal{H}$, and suppose for a contradiction that $|M| \leq m+q-1$. Also let~$D \subseteq [t]$ be the set of colours not used by~$M$, so $|D| = t - |M| \geq k$. Since $|V_i| \geq n \geq m+q+k-1$ for each $i \in [k]$, we may choose~$k$ pairwise-disjoint crossing $(k-1)$-tuples $t_1, \ldots, t_k$, each disjoint from~$V(M)$, such that~$t_i$ avoids~$V_i$ for each $i \in [k]$, as well as a crossing $k$-tuple~$f$ which is disjoint from~$V(M)$ and each of the~$t_i$. 
Let $F := \{j \in [t] : f \in E(H_j)\}$, so~$F$ is the set of colours that~$f$ could take in a rainbow matching. We have~$|F| \geq m$ by (\ref{lem:almostperfectrainbowmatching_ii}), and also~$F \cap D = \emptyset$ as otherwise we could add~$f$ in the corresponding colour to extend~$M$ to a larger rainbow matching. Let~$M' \subseteq M$ consist of all edges of~$M$ whose colour is not in~$F$, so $|M'| = |M| - |F| \leq q-1$.

For each $v \in V(M)$ let~$e_v$ denote the edge of~$M$ which contains~$v$. Now consider any $i \in [k]$, $j \in D$ and vertex $v \in N_{H_{j}}(t_i)$. We must have $v \in V(M)$, as otherwise we could extend~$M$ to a larger rainbow matching by adding the edge $t_i \cup \{v\}$ in colour~$j$ to obtain a contradiction. This means~$e_v$ is defined; moreover if~$e_v$ has colour~$c$ for some~$c \in F$, then $(M \setminus \{e_v\}) \cup \{f, \{v\} \cup t_i\}$ is a larger rainbow matching in~$\mathcal{H}$ (where~$f$ has colour~$c$ and $\{v\} \cup t_i$ has colour~$j$), again giving a contradiction. We conclude that for every $i \in [k]$ and $j \in D$ we have $e_v \in M'$ for every vertex $v \in N_{H_{j}}(t_i)$.

Since~$|D| \geq k$, we may choose distinct $d_1, \ldots, d_k \in D$ and note for each $i \in [k]$ that $N_{H_{d_i}}(t_i)$ has at least~$a_i$ neighbours in~$V_i$, each in an edge in~$M'$. Since $|M'| \leq q-1$, by~(\ref{lem:almostperfectrainbowmatching_i}) and our previous observation, there exist $i, i' \in [k]$  with $i \neq i'$ and $e \in M'$ such that $N_{H_{d_i}}(t_i) \cap e \neq \emptyset$ and $N_{H_{d_{i'}}}(t_{i'}) \cap e \neq \emptyset$. So there are vertices $v \in V_i$ and $v' \in V_{i'}$ with $v \in e \cap N_{H_{d_i}}(t_i)$ and $v' \in e \cap N_{H_{d_{i'}}}(t_{i'})$ respectively. It follows that $(M \setminus \{e\}) \cup \{t_i \cup \{v\}, t_{i'} \cup \{v'\}\}$, where $t_i \cup \{v\}$ and $t_{i'} \cup \{v'\}$ have colour~$d_i$ and~$d_{i'}$ respectively, is a larger rainbow matching in~$\mathcal{H}$, a contradiction. 

We conclude that~$\mathcal{H}$ admits a rainbow matching~$M$ of size~$m+q$. Suppose that there is some $j \in C$ for which~$M$ does not include an edge of colour~$j$. Fix~$i$ for which~$a_i$ is largest, so~$a_i \geq q/k$, and let~$f$ be a crossing $(k-1)$-tuple which avoids~$V_i$ and does not intersect~$V(M)$ (this is possible since~$n \geq m+q+1$). Since fewer than~$|C| \leq q/k$ vertices of~$V_i$ lie in edges of~$M$ with colours in~$C$, but $|N_{H_j}(f)| \geq a_i \geq q/k$, we may choose $v \in V_i$ for which $e:= f \cup \{v\} \in E(H_j)$ and either~$v \notin V(M)$ or~$v$ is contained in an edge~$e' \in M$ whose colour is not in~$C$. In the former case choose any edge~$e' \in M$ whose colour is not in~$C$. In both cases $(M \sm e') \cup \{e\}$, with~$e$ having colour~$j$, is then a rainbow matching of size~$m+q$ which, compared to~$M$, includes one more edge whose colour is in~$C$. Iterating this process gives the result.
\end{proof}

Our next lemma is a type of stability result: it states that if~$t$ is not too much larger than~$q$ then the multipartite degree assumption of Lemma~\ref{lem:almostperfectrainbowmatching} on its own is sufficient to guarantee a perfect rainbow matching in~$\mathcal{H}$ unless the family of $k$-graphs~$\mathcal{H}$ is extremal in the sense that almost all of the $k$-graphs in~$\mathcal{H}$ admit dominating sets of size close to~$q$ (note that the degree condition immediately implies that every dominating set must have size at least~$q$).

\begin{lemma} \label{lem:rainbow_or_domset}
Let $V_1, V_2, \dots, V_k$ be disjoint sets of vertices, each of size~$n$, and let $V := \bigcup_{i \in [k]} V_i$. Let $H_1, \dots, H_t$ be $k$-partite $k$-graphs each with vertex classes $V_1, \dots, V_k$, and suppose that for each $i \in [k]$ and $j \in [t]$ we have $\delta_{[k]\setminus\{i\}}(H_j) \geq a_i$. Write $q := \sum_{i \in [k]} a_i$.

For every $\eps \in (0, 1)$, if $t \leq (1+\eps)q$ and $n \geq 8k^3q/\eps$ then either
\begin{enumerate}[(i)]
    \item \label{lem:rainbow_or_domset_i} $\mathcal{H} :=\{H_1, \dots, H_t\}$ admits a perfect rainbow matching, or
    \item \label{lem:rainbow_or_domset_ii} there is a set $C \subseteq [t]$ of size at least~$(1-\eps)q$ such that for each~$j \in C$ the graph~$H_j$ admits a dominating set of size at most~$(1+2k\eps)q$.
\end{enumerate}
\end{lemma}

\begin{proof}
Suppose that property (\ref{lem:rainbow_or_domset_ii}) does not hold. By relabelling the~$H_i$, we may assume that the first $\ell := t-q + \lceil \eps q \rceil$ graphs $H_1, H_2, \dots, H_\ell$ each do not admit a dominating set of size at most~$(1+2k\eps)q$. We begin by proceeding through the graphs~$H_j$ in reverse order, setting aside for each either a vertex~$v_j$ or an edge~$e_j$, which we will subsequently use to construct a perfect rainbow matching in~$\mathcal{H}$. We also keep track of the sets~$U_j$ of all vertices in~$V$ which have not appeared as~$v_{j'}$ or in~$e_{j'}$ for any previously-considered~$j' > j$ (note that~$U_j$, like~$V$, may contain vertices from all vertex classes). So we set~$U_t := V$, and then for each $j = t, t-1, \dots, 1$ (in that order) we do the following.
\begin{enumerate}[(a)]
    \item \label{matchingpropa} If~$H_j[U_j]$ admits a matching of size~$kj$, then set~$U_{j-1} = U_j$.
    \item \label{matchingpropb} Otherwise, if~$H_j[U_j]$ contains a vertex~$x_j$ with $\deg_{H_j[U_j]}(x_j) \geq kjn^{k-2}$, then set~$U_{j-1} = U_j \sm \{x_j\}$.
    \item \label{matchingpropc} Otherwise, arbitrarily choose an edge~$e_j \in H_j[U_j]$ and set~$U_{j-1} = U_j \sm e_j$.
\end{enumerate}
The key claim is that we can do this for every~$j$; in other words, one of~(\ref{matchingpropa}),~(\ref{matchingpropb}) and~(\ref{matchingpropc}) will always be possible. To show this we first argue by induction on~$t$ that for each~$j \geq \ell$ either~(\ref{matchingpropa}) or~(\ref{matchingpropb}) will be possible, meaning that for each such~$j$ we have~$|V \sm U_j| \leq t-j$. If~$H_j[U_j]$ admits a matching of size~$kj$ then~(\ref{matchingpropa}) is possible, so we may assume that a largest matching in~$H_j[U_j]$ has size less than~$kj$, whereupon the vertex set of this matching is a dominating set~$D$ in~$H_j[U_j]$ of size at most~$k^2j$. For each~$i\in [k]$ choose a set of exactly~$\lceil n^{k-1}/2\rceil$ crossing $(k-1)$-tuples~$f$ which avoid~$V_i$ and have~$f \subseteq U_j \sm D$; this is possible since there are~$n^{k-1}$ crossing $(k-1)$-tuples avoiding~$V_i$ and $|(V \sm U_j) \cup D| \leq t+k^2t \leq n/3$. Each such~$f$ has at least~$a_i$ neighbours in~$V_i$, each of which must lie in~$D$ since~$D$ is a dominating set. Summing over each~$i$ we obtain $q \lceil n^{k-1}/2\rceil$ edges of~$H_j$ which contain a vertex of~$D$. Since $|V \sm U_j| \leq t-j$ at most $(t-j)\lceil n^{k-1}/2\rceil$
of these edges are not in~$H_j[U_j]$. So for $j \geq \ell = t - q + \lceil \eps q \rceil$ we find that there are at least $(q-t+j) \lceil n^{k-1}/2\rceil \geq \eps q n^{k-1}/2$ edges of~$H_j[U_j]$ which contain a vertex of~$D$. Since~$|D| \leq k^2j$, by an averaging argument we find that some vertex~$x_j \in D$ has 
$$\deg_{H_j[U_j]}(x_j) \geq \frac{\eps qn^{k-1}}{2k^2j} \geq kjn^{k-2},$$
where the final inequality holds since $j^2 \leq t^2 \leq (1+\eps)^2q^2 \leq 4 q^2$ and $n \geq 8k^3q/\eps$. We conclude that, as claimed,~(\ref{matchingpropb}) is possible.

It remains to consider the cases when $j \leq \ell - 1 \leq t - q + \eps q \leq 2\eps q$. For this, first note that since either~(\ref{matchingpropa}) or~(\ref{matchingpropb}) was chosen for each $j \geq \ell$, in each such case we have 
$$|V \sm U_j| \leq (t - \ell + 1) + k \cdot (\ell - 1)  = t + (k-1)(\ell-1) \leq (1+\eps)q + 2(k-1)\eps q \leq (1 + 2k \eps) q.$$ 
The fact that~$H_j$ has no dominating set of size at most~$(1+2k\eps) q$ then implies that~$V \sm U_j$ is not a dominating set in~$H_j$, and therefore we may choose an edge~$e_j \in H_j[U_j]$ as required for~(\ref{matchingpropc}).

To complete the proof we now use the chosen vertices and edges to build a rainbow matching~$M_t$ in~$\mathcal{H}$, giving property~(\ref{lem:rainbow_or_domset_i}). Initially set~$M_0 = \emptyset$; by adding edges one by one we will build a rainbow matching~$M_j$ of size~$j$ in $\{H_1[U_j], \dots, H_j[U_j]\}$ for each~$j \in [t]$ (note in particular that for each~$j \in [t]$ the matching~$M_j$ will only use vertices in~$U_j$). 

For $j=1, 2, \dots, t$:
\begin{enumerate}[(1)]
    \item If we chose an edge~$e_j$ in~$H_j[U_j]$, then~$e_j$ is already defined; observe that by definition of~$U_{j-1}$ we know that~$e_j$ is disjoint from~$V(M_{j-1})$.
    \item If we chose a vertex~$x_j$, then choose an edge~$e_j$ in~$H_j[U_j]$ which is disjoint from~$V(M_{j-1})$. This is possible since~$x_j \notin V(M_{j-1})$ by definition of~$U_{j-1}$, and furthermore we have $\deg_{H_j[U_j]}(x_j) \geq kjn^{k-2} > |V(M_{j-1})|n^{k-2}$, and the latter term is an upper bound on the number of edges containing~$x_j$ which meet~$V(M_{j-1})$.
    \item Otherwise~$H_j[U_j]$ admits a matching of size~$kj$; since $|V(M_{j-1})| = k(j-1) < kj$ we may choose an edge~$e_j$ in~$H_j[U_j]$ which is disjoint from~$V(M_{j-1})$.
    \item In each case we add~$e_j$ to~$M_{j-1}$ to obtain the desired rainbow matching~$M_j$ of size~$j$ in $\{H_1[U_j], \dots, H_j[U_j]\}$.
\end{enumerate}
At the conclusion of this process,~$M_t$ is a rainbow matching of size~$t$ in $\{H_1, \dots, H_t\}$.
\end{proof}

The next proposition helps us to use the dominating sets obtained by applying Lemma~\ref{lem:rainbow_or_domset} in the case where~$\mathcal{H}$ does not contain a perfect rainbow matching, by showing that almost all vertices in these dominating sets must have very high degrees in their host graphs.

\begin{proposition} \label{prop:domset_to_degree}
Let~$H$ be a $k$-partite $k$-graph with vertex classes $V_1, \dots, V_k$ each of size~$n$, such that for each~$i \in [k]$ we have $\delta_{[k]\setminus\{i\}}(H) \geq a_i$. Write $q := \sum_{i\in [k]} a_i$ and fix a real number~$\mu \geq 0$. If $n \geq 5 (1+\mu)q$ and~$H$ admits a dominating set~$D$ of size at most~$(1+\mu)q$, then at least~$(1-2\mu)q$ vertices~$x \in D$ have $\deg_{H}(x) \geq n^{k-1}/2$.
\end{proposition}

\begin{proof}
For~$q=0$ the statement is vacuous, so we may assume that~$q\geq 1$ and so~$n \geq 5$. Let $3/4 \leq \beta \leq 4/5$ be such that~$\beta n^{k-1}$ is an integer. For each $i \in [k]$ let~$F_i$ be a set of~$\beta n^{k-1}$ crossing $(k-1)$-tuples, each of which does not intersect~$D \cup V_i$. This is possible since there are~$n^{k-1}$ crossing $(k-1)$-tuples which avoid~$V_i$, and at most $|D| n^{k-2} \leq n^{k-1}/5$ of these intersect~$D$. Write $F = \bigcup_{i \in [k]} F_i$; we now double-count the set $X := \{(x, f): x \in D, f \in F, \{x\} \cup f \in E(H)\}$. First note that for each~$i \in [k]$, each~$f \in F_i$ has at least~$a_i$ neighbours in~$V_i$, all of which must lie in~$D$ since~$D$ is a dominating set which does not intersect~$f$. So 
$$|X| \geq \sum_{i \in [k]} a_i |F_i| = \sum_{i \in [k]} a_i \cdot \beta n^{k-1} = \beta qn^{k-1}.$$ Now let~$A \subseteq D$ consist of all vertices~$x \in D$ which lie in at least~$n^{k-1}/2$ pairs in~$X$. Since every vertex of~$D$ can be in at most~$\beta n^{k-1}$ pairs in~$X$, we then have 
\begin{align*} 
|X| & < \left(|D|-|A|\right) \cdot \frac{n^{k-1}}{2} + |A| \cdot \beta n^{k-1} = \frac{|D|n^{k-1}}{2} + (\beta - 1/2)|A|n^{k-1}  \\ &\leq \frac{(1+\mu)qn^{k-1}}{2} + (\beta - 1/2)|A|n^{k-1}.
\end{align*}
Combining the two displayed inequalities and noting that~$1/(2\beta -1) \leq 2$, we obtain~$|A| \geq (1-2\mu)q$, and the conclusion follows.
\end{proof}

The final proposition we need concerns perfect matchings in bipartite graphs. Let~$G$ be a bipartite graph with vertex classes~$A$ and~$B$, where every vertex in~$A$ has degree close to~$|A|$ (we don't control the size of~$B$ at all beyond this, or place any degree condition on vertices in~$B$). The idea expressed in this proposition is that either the neighbourhoods of vertices in~$A$ do not mostly overlap extensively, in which case there should be a matching in~$G$ of size~$|A|$ even after the deletion of any small set of vertices of~$B$, or they do mostly overlap extensively, in which case we can remove a few vertices of~$A$ to obtain a subset~$X \subseteq A$ for which some usefully-sized subset~$B' \subseteq B$ is `highly-matchable' to~$X$, meaning that every subset~$A' \subseteq X$ of size~$|B'|$ can be perfectly matched to~$B'$.

\begin{proposition} \label{prop:bip_matching}
Fix a positive integer~$q$ and a real number~$\mu$ with $0 \leq \mu \leq 1/50$. Let~$G$ be a bipartite graph with vertex classes~$A$ and~$B$, where~$|A| \leq (1+\mu)q$, and suppose that $\deg(x) \geq (1-\mu)q$ for every~$x \in A$. We then must have either
\begin{enumerate}[(i)]
    \item\label{prop:bip_matching_i} for every set~$Y \subseteq B$ with~$|Y| \leq \mu q$ there exists a matching of size~$|A|$ in~$G[A \cup (B \sm Y)]$, or
    \item\label{prop:bip_matching_ii} there exist subsets~$X \subseteq A$ and~$B' \subseteq B$ with~$|X| \geq (1-2\mu)q$ and~$|B'| = \lfloor q/2 \rfloor$ such that for every subset~$A' \subseteq X$ of size~$\lfloor q/2 \rfloor$ there is a perfect matching in~$G[A' \cup B']$.
\end{enumerate}
\end{proposition}

\begin{proof}
    We may assume that there exists~$Y \subseteq B$ with~$|Y| \leq \mu q$ for which there does not exist a matching of size~$|A|$ in~$G[A \cup (B \sm Y)]$, as otherwise we obtain (\ref{prop:bip_matching_i}). By Hall's Marriage Theorem it follows that some non-empty~$X \subseteq A$ has~$|N_X \sm Y| < |X|$, where we write~$N_X := N_G(X)$. 
    So $(1-\mu)q \leq |N_X| < |X| + |Y| \leq |A|+|Y| \leq (1+2\mu)q$, where the first inequality holds since each vertex~$x \in A$ has $\deg(x) \geq (1-\mu)q$. This in turn implies that $(1-2\mu)q \leq (1-\mu) q - |Y| < |X| \leq |A| \leq (1+\mu)q$.

    For each~$x \in X$, since~$|N_X| \leq (1+2\mu)q$ and~$\deg(x) \geq (1-\mu)q$, Observe that for each~$x \in X$ there are at most~$3\mu q$ vertices~$y \in N_X$ for which~$xy$ is not an edge. So the number of pairs~$xy$ with~$x \in X$ and~$y \in N_X$ for which~$xy$ is not an edge is at most $3\mu q|X|\leq 4\mu q^2$. It follows that at most~$20\mu q$ vertices of~$N_X$ lie in more than~$q/5$ such pairs. So we may choose a set~$B' \subseteq N_X \subseteq B$ of size~$|B'| = \lfloor q/2 \rfloor$ not including any such vertex, meaning that every~$y \in B'$ has at least~$|X| - q/5$ neighbours in~$X$.

    Now consider any set~$A' \subseteq X$ of size~$\lfloor q/2 \rfloor$. Our choice of~$B'$ ensures that each vertex of~$B'$ has at least $|A'| - q/5 \geq |A'|/2$ neighbours in~$A'$, whilst each vertex~$x \in A'$ has at least $|B'| - 3\mu q \geq |B'|/2$ neighbours in~$B'$ (the calculations of this sentence actually only hold for~$q \geq 4$, but the proposition is elementary to check for smaller~$q$). It follows that~$G[A' \cup B']$ admits a perfect matching, so we have (\ref{prop:bip_matching_ii}).
\end{proof}

Finally we arrive at the main result of this section, which is the key step in the proof of Theorem~\ref{thm_extremal_n-1}.

\begin{lemma} \label{lem:perfect_rainbow_matching}
Let $V_1, V_2, \dots, V_k$ be disjoint sets of vertices each of size~$n$, and let $V = \bigcup_{i \in [k]} V_i$. Let $H_1, \dots, H_t$ be $k$-partite $k$-graphs each with vertex classes $V_1, \dots, V_k$, and suppose that
\begin{enumerate}[(i)]
    \item \label{lem:perfect_rainbow_matching_i} for each~$i \in [k]$ and~$j \in [t]$ we have $\delta_{[k]\setminus\{i\}}(H_j) \geq a_i$, and
    \item \label{lem:perfect_rainbow_matching_ii} each crossing $k$-tuple appears in at least~$m$ of the $k$-graphs~$H_j$.
\end{enumerate}
Also write $q := \sum_{i \in [k]} a_i$. If $n \geq \max(1600k^4q, 100k^6)$ and $m+q \leq t \leq (1+1/200k)q$, then $\mathcal{H} :=\{H_1, \dots, H_t\}$ admits a rainbow matching of size~$m+q$.
\end{lemma}

\begin{proof}
Suppose first that $\lfloor{q/2}\rfloor \geq k-1$ (we address the other, simpler, case later). Set $\eta := 1/200k$, so $t \leq (1+\eta)q$ and $n \geq 8k^3q/\eta$. These bounds permit us to apply Lemma~\ref{lem:rainbow_or_domset}, by which we may assume that there exists a set~$A \subseteq [t]$ of size at least~$(1-\eta)q$ such that for each~$j \in A$ the graph~$H_j$ admits a dominating set of size at most~$(1+2k\eta)q$, and by Proposition~\ref{prop:domset_to_degree} it follows that for each~$j \in A$ there exists a set~$X_j \subseteq V$ of size $|X_j| \geq (1-4k\eta)q$ such that each vertex~$x \in X_j$ has $\deg_{H_j}(x) \geq n^{k-1}/2$. Form an auxiliary bipartite graph~$G$ whose vertex classes are~$A$ and~$V$ in which~$jx$ is an edge if $\deg_{H_j}(x) \geq n^{k-1}/2$. So our prior observation implies that $\deg_G(j) \geq (1-4k\eta)q$ for each~$j \in A$. By Proposition~\ref{prop:bip_matching}, taking $\mu := 4k\eta = 1/50$ and $s:= \lfloor q/2 \rfloor$, we find that either
\begin{enumerate}[(a)]
    \item~\label{firstoutcome} for every set~$Y \subseteq V$ with~$|Y| \leq \mu q$ there exists a matching of size~$|A|$ in~$G[A \cup (V \sm Y)]$, or
    \item~\label{secondoutcome} there exist subsets~$X \subseteq A$ and~$Z \subseteq V$ with~$|X| \geq (1-2\mu)q$ and~$|Z| = s$ such that for any~$A' \subseteq X$ of size~$s$ there is a perfect matching in~$G[A' \cup Z]$.
\end{enumerate}

If~(\ref{firstoutcome}) holds, then arbitrarily choose a perfect rainbow matching~$M$ in the family~$\{H_j : j \in [t] \sm A\}$. This is possible since $|[t] \sm A| \leq 2\eta q \leq q/k \leq \max_{i \in [k]} a_i$, so we may form~$M$ by choosing~$t - |A|$ pairwise-disjoint crossing~$(k-1)$-tuples which avoid~$V_{i'}$, where~$i'$ is such that~$a_{i'} :=\max_{i \in [k]} a_i$, and greedily extending each to an edge of the respective~$H_j$ so that no vertex is used more than once. Let~$Y := V(M)$, so $|Y| = k|M| \leq 2k\eta q \leq \mu q$. It follows by~(\ref{firstoutcome}) that~$G[A \cup (V \sm Y)]$ admits a matching~$M_G$ which covers the vertices of~$A$, which we will shortly use to greedily extend~$M$ to a perfect rainbow matching in~$\mathcal{H}$. 

If instead~(\ref{secondoutcome}) holds, then set $V_i' := V_i \setminus Z$ for each $i \in [k]$, $V' := V \setminus Z$, $n' := n - s$, $q' = q - s$ and $H'_j := H_j[V']$ for each $j \in [t]$. Then $H'_1, \dots, H'_t$ are $k$-partite $k$-graphs with common vertex classes $V'_1, \dots, V'_k$ each of size at least~$n'$ such that each crossing $k$-tuple appears in at least~$m$ of the graphs $H'_1, \dots, H'_t$. Moreover, for each $i \in [k]$ and $j \in [t]$ we have $\delta_{[k]\setminus\{i\}}(H_j) \geq a'_i := a_i - |Z \cap V_i|$, and $\sum_{i \in [k]} a'_i = \sum_{i \in [k]} a_i - |Z| = q'$. Since $t\geq m+q =m+q'+s \geq m+q'+k-1$, we may apply Lemma~\ref{lem:almostperfectrainbowmatching} with $[t] \sm X$ in place of~$C$ to obtain a rainbow matching~$M$ of size $m+q' = m+q-s$ in the family $\{H'_1, \dots, H'_t\}$ such that~$M$ includes an edge of colour~$j$ for every $j \in [t] \sm X$ (i.e.~$M$ includes distinct edges for each such~$j$). Let~$U \subseteq [t]$ be the set of all~$j \in [t]$ for which~$M$ has no edge of colour~$j$, 
so~$U \subseteq X$ and furthermore $|U| = t-(m+q-s) \geq s$. So we may arbitrarily choose a subset~$A' \subseteq U \subseteq X$ of size~$s$, whereupon~(\ref{secondoutcome}) implies that there is a perfect matching~$M_G$ in~$G[A' \cup Z]$, which we will use to extend~$M$ to a rainbow matching in~$G$ of size~$m+q$.

In both cases we now use the matching~$M_G$ to greedily extend~$M$ to a rainbow matching of size~$|M| + |M_G|$ in~$\mathcal{H}$. Observe that each edge~$e \in M_G$ has the form~$j_ex_e$ for some~$j_e \in A$ and~$x_e \in V$. For each edge~$e \in M_G$ in turn we choose an edge~$f_e$ of~$H_{j_e}$ with~$x_e \in f_e$ such that~$f_e$ does not intersect~$V(M)$ or any previously-chosen edge~$f_{e'}$ and also~$f_e$ does not contain any vertex~$x_{e'}$ for~$e' \neq e$. Observe for this that, since~$e$ is an edge of~$G$, there are at least $\deg_{H_{j_e}}(x_e) \geq n^{k-1}/2$ edges of~$H_{j_e}$ containing~$x_e$, and the number of these edges which we are forbidden from using is at most $(|V(M)| + k|M_G|)n^{k-2} \leq ktn^{k-2}$, so it is always possible to choose an edge~$f_e$ with the required properties. The chosen edges~$f_e$ then form a perfect rainbow matching~$M'$ of size~$|M_G|$ in the family~$\{H_{j_e} : e \in M_G\}$, and so~$M^* := M \cup M'$ is a rainbow matching in~$\mathcal{H}$ of size~$|M| + |M_G|$. It follows that~$M^*$ is perfect if~(\ref{firstoutcome}) holds, and has size~$m+q$ if~(\ref{secondoutcome}) holds; either way this completes the proof.

It remains to consider the case when~$\lfloor{q/2}\rfloor < k-1$, in which case we have~$q < 2k$, so our assumption that $m+q \leq t \leq (1+1/200k)q$ implies that~$m=0$ and~$t = q$; in particular, our goal in this case is to show that~$\mathcal{H}$ admits a perfect rainbow matching. Set~$\eps = 1/3kq$ and observe that we have $n \geq 100k^6 \geq 25 k^4q^2 \geq 8k^3q/\eps$, so again we may apply Lemma~\ref{lem:rainbow_or_domset}, by which we may assume that there exists a set~$A \subseteq [t]$ of size~$q=t$ such that for each~$j \in A$ the graph~$H_j$ admits a dominating set of size at most~$q$. In other words, for every~$j \in [t]$ there exists a dominating set~$S_j$ in~$H_j$ of size at most~$q$. Write $X := \bigcup_{j \in [t]} S_j$, so $|X| \leq qt = q^2$, and define~$X_i := X \cap V_i$ for each~$i \in [k]$. We now build our desired matching greedily; for this we set~$M_0 := \emptyset$, and for each~$j \in [t]$ in turn will add an edge to~$M_{j-1}$ to form a perfect rainbow matching~$M_j$ in~$\{H_1, \dots, H_j\}$, whilst maintaining the property that every edge in~$M_j$ has at most one vertex in~$X$. So fix some~$j \in [t]$, and note that since each edge in~$M_{j-1}$ has at most one vertex in~$X$, we have $\sum_{i \in [k]} a_i = q = t > j-1 = |V(M_{j-1}) \cap X| = \sum_{i \in [k]} |V(M_{j-1}) \cap X_i|$, and so there exists some~$i \in [k]$ with $|V(M_{j-1}) \cap X_i| < a_i$. Choose a crossing $(k-1)$-tuple~$f_j$ of vertices in $V \sm (V(M_{j-1}) \cup X)$ which avoids~$V_i$ (this is possible since $n > kq + q^2 \geq |V(M_{j-1}) \cup X|$), and observe that~$f$ has at least~$a_i$ neighbours in~$V_i$, all of which must lie in~$S_j \subseteq X$, so there is at least one vertex~$u_j \in X_i \sm V(M_{j-1})$ for which~$e_j = f_j \cup \{u_j\}$ is an edge of~$H_j$. Set~$M_j := M_{j-1} \cup \{e_j\}$, and observe that~$M_j$ is a perfect rainbow matching in $\{H_1, \dots, H_j\}$ in which every edge has at most one vertex in~$X$. Having done this for every~$j \in [t]$, the final matching~$M_t$ is the desired perfect rainbow matching in~$\mathcal{H}$.
\end{proof}

\section{Proof of Theorem~\ref{thm_extremal_n-1}} \label{sec:proof}

We begin our proof of Theorem~\ref{thm_extremal_n-1} with two propositions regarding the link graphs of vertices in the first vertex class~$V_1$. The first states that for large sets~$U \subseteq V_1$, every crossing $(k-1)$-tuple avoiding~$V_1$ must appear in many of the link graphs~$\{L_v(H)\}_{v \in U}$. 

\begin{proposition}\label{prop_complete}
Let~$H$ be a $k$-partite $k$-graph whose vertex classes $V_1, \dots, V_k$ each have size~$n$, with the property that $\delta_{[k]\setminus\{1\}}(H)\geq a_1$. For every set~$U \subseteq V_1$ and each crossing $(k-1)$-tuple~$f$ avoiding~$V_1$, there are at least~$|U|-(n-a_1)$ vertices~$v \in U$ for which~$f \in L_v(H)$.
\end{proposition}

\begin{proof}
Observe that~$f$ has at least~$a_1$ neighbours in~$V_1$, of which at most~$n-|U|$ are outside~$U$, meaning that~$f$ has at least~$a_1 - (n-|U|)$ neighbours~$v \in U$, and~$f$ appears in~$L_v(H)$ for each such~$v$.
\end{proof}

Our second proposition asserts that for any matching~$M$ in the complete $(k-1)$-partite $(k-1)$-graph on $V_2, \ldots, V_k$, there is a large subset~$U \subseteq V_1$ such that~$L_v(H)$ contains many edges of~$M$ for every~$v \in U$. In fact we give two forms of this statement, varying in how we interpret `large' and `many'. The first version is a looser result which holds for a wider range of values of~$q$, which will be sufficient for the proof of Theorem~\ref{thm_extremal_n-1} in the case where~$q \geq 400 k^2$. For smaller values of~$q$ we instead use the second version of the statement, which gives a more precise result but which is only valid for~$q \leq \sqrt{\frac{n}{k+1}}$ (we will in fact only use it in the case where~$q < 400 k^2$).

\begin{proposition}\label{lemma_extremal_count}
For each $k \geq 3$ and every $\eps \in (0,1)$ there exists $n_0 \in \mathbb{N}$ such that the following statements hold. Let~$H$ be a $k$-partite $k$-graph with vertex classes $V_1, \ldots, V_k$ each of size~$n \geq n_0$ in which $\delta_{[k]\setminus\{i\}}(H) \geq a_i$ for each~$i \in [k]$. Suppose also that $a_1 \geq a_2 \geq \ldots \geq a_k$ and $\sum_{i \in [k]} a_i > n-k$, and write $q:=\sum_{i \in [2,k]} a_i$. Let~$M$ be a perfect matching in the complete $(k-1)$-partite $(k-1)$-graph on vertex sets $V_2, \ldots, V_k$.
\begin{enumerate}[(i)]
    \item \label{anyq} If $2k/\eps \leq q \leq \frac{\eps n}{8k}$, then there exists a subset~$U \subseteq V_1$ with~$|U| \geq n-(1+\eps)q$ such that every~$u \in U$ has~$|L_u(H) \cap M| \geq (1+\eps)kq$.
    \item  \label{smallq} If $q \leq \sqrt{\frac{n}{k+1}}$, then there exists a subset~$U \subseteq V_1$ with~$|U| \geq a_1$ such that every~$u \in U$ has~$|L_u(H) \cap M| \geq n-a_1 +(k-1)q$.
\end{enumerate}
\end{proposition}

\begin{proof}
Let~$P$ be the set of pairs~$(u, f)$ with~$u \in V_1$ and~$f \in M$ such that~$\{u\} \cup f \in E(H)$. Since~$|M| = n$ and each~$f \in M$ has at least~$a_1$ neighbours in~$V_1$ we have~$|P| \geq a_1n$. So, if we suppose for a contradiction that~(\ref{anyq}) fails, then we have 
$$(n-k-q)n < a_1 n \leq |P| < (n-(1+\eps)q) \cdot n + (1+\eps)q \cdot (1+\eps)kq < (n-q)n - \eps q n + 4kq^2,$$
where the first inequality holds since $a_1 + q = \sum_{i \in [k]} a_i > n-k$, by assumption. It follows that $kn+4kq^2 > \eps qn$, from which a contradiction arises because our bounds on~$q$ give $kn \leq \eps qn/2$ and $4kq^2 \leq \eps qn/2$. We conclude that~(\ref{anyq}) holds.

Similarly, if we suppose for a contradiction that~(\ref{smallq}) fails, then we have
$$a_1n \leq |P| \leq (a_1-1)\cdot n + (n-a_1+1) \cdot (n-a_1+(k-1)q) < a_1n - n + (q+k)(kq+k),$$
where the final inequality holds since $n-a_1 < k+q$ (and each of these variables are integers). It follows that $k(q+k)(q+1) > n$, but this contradicts our assumption that $q \leq \sqrt{\frac{n}{k+1}}$ (for~$n$ sufficiently large with respect to~$k$). We conclude that~(\ref{smallq}) holds also.
\end{proof}

We can now give the proof of Theorem~\ref{thm_extremal_n-1}, which combines Propositions~\ref{prop_complete} and~\ref{lemma_extremal_count} with Lemma~\ref{lem:perfect_rainbow_matching} before finishing with an application of Hall's Theorem to an appropriate auxiliary bipartite graph to obtain a matching in~$H$ of the desired size.

\begin{proof}[Proof of Theorem~\ref{thm_extremal_n-1}]
By reducing~$a_1$ if necessary, we may assume that $\sum_{i \in [k]} a_i \leq n$, so our goal is to show that~$H$ contains a matching of size~$\sum_{i \in [k]} a_i$. If $\sum_{i \in [k]} a_i \leq n-k+2$ then we obtain a matching of this size by Fact~\ref{prop_fact}, so we may assume that $\sum_{i \in [k]} a_i \geq n-k+3$ also.  
Define $q:= \sum_{i \in [2,k]} a_i$, so $n-k+3 - a_1 \leq q \leq n-a_1$, and by assumption we have $q \leq n/1600k^4$. Also let~$K^{k-1}$ denote the complete $(k-1)$-partite $(k-1)$-graph on vertex sets $V_2, \ldots, V_k$, and arbitrarily choose a matching~$M$ of size~$n$ in~$K^{k-1}$.

Suppose first that~$q \geq 400k^2$. Applying Proposition~\ref{lemma_extremal_count} with~$\eps = 1/200k$ we obtain a subset~$X \subseteq V_1$ such that
\begin{enumerate}[(a)]
\item \label{Xpropi} $n-(1+\eps)q \leq |X|$, and
\item \label{Xpropii} $|M \cap L_u(H)| \geq (1+\eps)kq$ for every~$u \in X$.
\end{enumerate}
Since $a_1 \geq n-q-k+3 \geq n-(1+\eps)q + 1$ we may additionally insist that~$|X| \leq a_1$ (by removing vertices from~$X$ if necessary).
Let~$Y:= V_1 \setminus X$, so $n-a_1 \leq |Y| \leq (1+\eps)q$  by~(\ref{Xpropi}). Applying Proposition~\ref{prop_complete} to~$Y$ we find that each crossing $(k-1)$-tuple avoiding~$V_1$ appears in at least~$m := |Y|-(n-a_1) \geq 0$ of the link graphs~$\{L_u(H)\}_{u \in Y}$. Since~$q \leq n-a_1$ we have~$|Y| \geq m+q$, so by Lemma~\ref{lem:perfect_rainbow_matching}, applied with~$m$ and~$q$ as here and with~$t = |Y|$, the family~$\{L_u(H)\}_{u \in Y}$ admits a rainbow matching~$M_{rainbow}$ of size~$m+q \leq |Y|$.
This means there is a set~$Z \subseteq Y$ of size~$m+q$ and a bijection~$\psi: Z \to M_{rainbow}$ such that for each~$z \in Z$ we have~$\{z\} \cup \psi(z) \in E(H)$. 

Let~$M' \subseteq M$ be the matching in~$K^{k-1}$ consisting of all edges of~$M$ which do not intersect any edge of~$M_{rainbow}$, and let~$M^*$ be a matching in~$K^{k-1}$ of size~$n$ with~$M' \cup M_{rainbow} \subseteq M^*$. Since $|M \sm M'| \leq |V(M_{rainbow})|  
\leq (k-1)|Y| \leq (k-1)(1+\eps)q$, by~(\ref{Xpropii}) we have for each~$u \in X$ that 
\begin{align*}
|M^* \cap L_u(H)| &\geq |M' \cap L_u(H)| \geq |M \cap L_u(H)| - |M \setminus M'| 
\\ &\geq (1+\eps)kq - (k-1)(1+\eps)q \geq (1+\eps)q.
\end{align*}
 Set $U^* := X \cup Z \subseteq V_1$, and form an auxiliary bipartite graph~$B$ with vertex classes~$U^*$ and~$M^*$ in which, for each~$u \in U^*$ and~$f \in M^*$, we have~$uf \in E(B)$ if and only if $\{u\} \cup f \in E(H)$.

Consider a set~$S \subseteq U^*$. If~$|S| \geq n-a_1+1$, then by Proposition~\ref{prop_complete} each~$f \in K^{k-1}$ has~$|N_H(f) \cap S| \geq 1$, so~$N_B(S) = M^*$ and so~$|N_B(S)| = n \geq |S|$. 
Alternatively, if~$|S| \leq n-a_1$ and~$S$ contains a vertex~$u \in X$, then $|N_B(S)| \geq |N_B(u)| = |M^* \cap L_u(H)| \geq (1+\eps)q \geq n-a_1 \geq |S|$. 
Finally, if~$S \subseteq Z$ then~$z \psi(z) \in E(B)$ for each~$z \in S$, so again we have~$|N_B(S)| \geq |S|$. 
We conclude that the bipartite graph~$B$ satisfies Hall's criterion that~$|N_B(S)| \geq |S|$ for every~$S \subseteq U^*$, and therefore contains a matching~$M_B$ of size~$|U^*|$. 
The edges~$\{u\} \cup f$ for each~$uf \in M_B$ then form a matching in~$H$ of size 
$$|U^*| = |X| + |Z| = n-|Y| + m + q = a_1+q = \sum_{i \in [k]} a_i,$$ as required. 

Now suppose instead that~$q<400k^2$. Our argument for this case is essentially the same as for the previous case, except that we use the more precise form of Proposition~\ref{lemma_extremal_count} (since~$q$ is now small enough to permit this). 
Indeed, Proposition~\ref{lemma_extremal_count} yields a set~$X \subseteq V_1$ with~$|X| = a_1 \leq n-q$ such that for every~$u \in X$ we have $|L_u \cap M| \geq n-a_1 +(k-1)q$. Let $Z \subseteq V_1 \setminus X$ have~$|Z| = q$. 
By Lemma~\ref{lem:perfect_rainbow_matching}, applied with~$m=0$ and~$t=q$, the family~$\{L_u(H)\}_{u \in Z}$ admits a perfect rainbow matching~$M_{rainbow}$. 
Define~$M' \subseteq M$ by $M' := \{e \in M : e \cap V(M_{rainbow}) = \emptyset\}$, and let~$M^*$ be a perfect matching in~$K^{k-1}$ with~$M' \cup  M_{rainbow} \subseteq M^*$. Since $|V(M_{rainbow})| = (k-1)|M_{rainbow}|= (k-1)q$, we have for each~$u \in X$ that $|L_u(H) \cap M^*| \geq |L_u(H) \cap M'| \geq n-a_1+(k-1)q-(k-1)q = n-a_1$. Set~$U^* = X \cup Z$, and form an auxiliary bipartite graph~$B$ with vertex classes~$U^*$ and~$M^*$ in which, for each~$u \in U^*$ and~$f \in M^*$, we have~$uf \in E(B)$ if and only if~$\{u\} \cup f \in E(H)$. 
Now consider a subset~$S \subseteq U^*$. If~$|S| \geq n-a_1+1$ then~$|N_B(S)| = n \geq |S|$ by Proposition~\ref{prop_complete}; alternatively, if~$|S| \leq n-a_1$ and~$S$ contains some~$u \in X$ then $|N_B(S)| \geq |N_B(u)| \geq n-a_1 \geq |S|$; finally if~$|S| \subseteq Z$ then $|N_B(S)| \geq |N_{M_{rainbow}}(S)| =|S|$. 
We conclude that the graph~$B$ satisfies Hall's criterion, and therefore contains a matching~$M_B$ of size~$|U^*|$; we then have a matching $\{\{u\} \cup f: u \in U^*, f \in M^*, uf \in M_B\}$ in~$H$ of size~$|U^*| = |X|+|Z| = a_1+q = \sum_{i \in [k]} a_i$ as required.
\end{proof}

\section{Concluding Remarks}

The main contribution of this paper is to extend Fact~\ref{thm_for_graphs} to the $k$-graph setting, with the degree condition for graphs being replaced by a codegree condition for $k$-graphs. Following Pikhurko~\cite{pikhurko} and Lo and Markstr\"{o}m~\cite{lo-mark}, it would be interesting to investigate lower bounds on the matching number~$\nu(H)$ for $k$-partite $k$-graphs~$H$ satisfying analogous minimum multipartite $\ell$-degree conditions for~$\ell < k-1$. 

As described in Section~\ref{sec:ideas}, we see the combination of degree and multiplicity conditions appearing in Lemma~\ref{lem:perfect_rainbow_matching} as being a natural setting to consider, and expect that result should hold for a much wider range of values of~$m, q$ and~$t$; we think this question is worthy of further attention.

\section*{Acknowledgements}

We thank Alexey Pokrovskiy for helpful discussions in initial stages of the project, and for communicating Lemma~\ref{lem_alexey} to us. We also note that our proof of Theorem~\ref{thm_extremal_n-1} -- particularly the reformulation in terms of rainbow matchings -- was inspired by Aharoni and Haxell's generalisation of Hall's theorem for hypergraphs~\cite{hyphall}, though this does not appear in the final version of our arguments.

We thank the anonymous reviewers for their helpful suggestions for improvements to the manuscript.

\end{document}